\documentclass[12pt,dvips]{article}
\usepackage[T2A]{fontenc}
\usepackage[utf8]{inputenc}
\usepackage{amssymb,amsmath,amsfonts,amsthm,amscd,latexsym,verbatim,indentfirst,geometry,xcolor}
\newtheorem{theorem}{Theorem}[section]
\theoremstyle{definition}

\newtheorem{corollary}[theorem]{Corollary}
\newtheorem{lemma}[theorem]{Lemma}
\newtheorem{proposition}[theorem]{Proposition}
\newtheorem{question}[theorem]{Question}

\newtheorem{definition}[theorem]{Definition}
\newtheorem{example}[theorem]{Example}
\newtheorem{remark}[theorem]{Remark}
\numberwithin{equation}{section}

\def\id{\mathrm{id}}

\def\op{\mathrm{op}}

\def\Conj{\mathrm{Conj}}

\def\Hol{\mathrm{Hol}}
\def\Soc{\mathrm{Soc}}
\def\Aut{\mathrm{Aut}}
\def\Out{\mathrm{Out}}
\def\Ann{\mathrm{Ann}}
\def\Imm{\mathrm{Im}\,}
\def\Inn{\mathrm{Inn}}

\begin{document}

\sloppy

\hfill{16T25, 20N99 (MSC2020)}

\begin{center}
{\Large
Rota---Baxter groups, skew left braces, \\
and the Yang---Baxter equation}

\smallskip

Valeriy G. Bardakov, Vsevolod Gubarev
\end{center}

\begin{abstract}
Braces were introduced by W. Rump in 2006 as an algebraic system
related to the quantum Yang---Baxter equation.
In 2017, L. Guarnieri and L. Vendramin
defined for the same purposes a~more general notion of a~skew left brace.
In 2021, L. Guo, H.~Lang and Y. Sheng gave a definition of what is a Rota---Baxter operator on a group.
We connect these two notions as follows.
It is shown that every Rota---Baxter group gives rise to a skew left brace.
Moreover, every skew left brace can be injectively embedded into a Rota---Baxter group.
When the additive group of a skew left brace is complete, then this brace is induced
by a Rota---Baxter group.
We interpret some notions of the theory of skew left braces in terms of Rota---Baxter operators.

{\it Keywords}:
Rota---Baxter operator, Rota---Baxter group, skew left brace, Yang---Baxter equation.
\end{abstract}

\section{Introduction}

The Yang---Baxter equation from mathematical physics has been studied since 1960s--1970s,
and there exist a~lot of quite different versions of it.
In 1992, V.G. Drinfeld~\cite{Drinfeld} highlighted
the importance of the study of set-theoretical solutions
to the quantum Yang---Baxter equation.

Recall that a~set-theoretical solution to the (quantum) Yang---Baxter equation
on a~set~$X$ is a~bijective map $S\colon X\times X\to X\times X$ such that
$$
(S \times \id) (\id \times S) (S \times \id)
 = (\id \times S) (S \times \id) (\id \times S).
$$
Note the important works of P. Etingof, T. Schedler, and A.~Soloviev~\cite{ESS,Solovev} 
and J.~Lu, M.~Yan and Y.~Zhu~\cite{LYZ00} in this direction. 
Actually, set-theoretic solutions to the Yang---Baxter equation
were studied before V.G.~Drinfeld formulated his question.
In 1980s, D.~Joyce \cite{Joyce} and S.~Matveev~\cite{Matveev} introduced
quandles as invariants of knots and links,
and every quandle gives a~set-theoretic solution to the Yang---Baxter equation.

In 2006, W. Rump introduced~\cite{Rump2,Rump} braces to study involutive set-theoretical solutions
to the Yang---Baxter equation, although the same algebraic objects were already considered by A.G.~Kurosh~\cite{Kurosh} in 1970s. 
In~2014, this notion was reformulated by
F.~Ced\'{o}, E.~Jespers, and J.~Okni\'{n}ski in~\cite{CedoJespersOkninski}.
In 2017, L. Guarnieri and L. Vendramin defined~\cite{GV2017} skew left braces
which give non-involutive solutions to the Yang---Baxter equation.

A set~$A$ with two binary operations $\cdot$ and $\circ$ is called a~skew left brace,
if $(A,\cdot)$ and $(A,\circ)$ are groups and the identity
$$
a \circ (b \cdot  c) =  (a \circ b) \cdot a^{-1} \cdot (a \circ c)
$$
holds for all $a,b,c\in A$.

In the last few years, theory of braces and skew left braces is a~dynamically
developed area with a lot of connections which include knot theory, Hopf---Galois theory,
and radical rings.

The main aim of the work is to study the deep connection between skew left braces 
and Rota---Baxter groups introduced in 2020 by L. Guo, H. Lang, Y.~Sheng~\cite{Guo2020}.
A~Rota---Baxter group is a group~$G$ endowed with a~map $B\colon G\to G$
satisfying the identity
$$
B(g)B(h) = B( g B(g) h B(g)^{-1} ),
$$
where $g,h\in G$.

This notion appeared as a group analogue of Rota---Baxter operators
defined on an algebra.
Rota---Baxter operators on algebras are known since the middle of the previous
century~\cite{Baxter,Tricomi} and they have in turn connections
with mathematical physics (classical and associative Yang---Baxter equations),
number theory, operad theory, Hopf algebras,
combinatorics et cetera, see the monograph~\cite{GuoMonograph}.

After the initial work~\cite{Guo2020},
the study of Rota---Baxter groups have been continued in~\cite{BG,CS2022,Goncharov2020,JSZ}.

We show that given a Rota---Baxter group~$(G,\cdot,B)$ one obtains a~skew left brace
$(G,\cdot,\circ_B)$, where $x \circ_B y = xB(x)yB(x)^{-1}$.
Let $(G,\cdot,\circ)$ be a~skew left brace satisfying the condition that
the~group $(G,\cdot)$ is complete,
then there exists a Rota---Baxter operator on $(G,\cdot)$ such that
$x \circ y = x \circ_B y$.
Moreover, we prove that every skew left brace can be (injectively) embedded into
a~Rota---Baxter group. The proof and the ideology which is behind it come
from the analogous result stated for Rota---Baxter algebras
and so called postalgebras~\cite{Embedding}.
Further, we apply the enveloping Rota---Baxter group to define left center and strong left nilpotency of skew left braces.

We apply the constructions of Rota---Baxter groups~\cite{BG,Guo2020}
to provide both known and new constructions of skew left braces.
Recently C. Tsang studied~\cite{Tsang2} skew left braces
$(A,\cdot,\circ)$ with $(A,\cdot) = G^n$ and $(A,\circ)\cong G^n$,
where $G$ is a~non-abelian simple finite group.
We discuss the connection between the results of C. Tsang and the construction of
Rota---Baxter operators on $G^n$ from~\cite{BG}.

We state that differentiation (in the definite sense) of a skew left brace which both groups are Lie ones
is a~post-Lie algebra.
It is shown how the free skew left brace comes from the free Rota---Baxter group.

We get the explicit formula of the non-degenerate
set-theoretical solution to the Yang---Baxter equation
arisen from a Rota---Baxter group~$(G,\cdot,B)$ via its induced skew left brace,
$$
S \colon G \times G \to G\times G,\quad
S(a, b) = (\lambda_a(b), a^{\lambda_a(b)B(\lambda_a(b))}),
$$
where $\lambda_a(b) = B(a)b B(a)^{-1}$ for $a,b \in G$.

\newpage

Let us give short {\bf outline} of the work.
In~\S2, we state the required preliminaries on Rota---Baxter groups
and skew left braces.

In~\S3, we study the connection between Rota---Baxter groups
and skew left braces. 
In~\S4, we apply constructions of Rota---Baxter groups
to get skew left braces.

In~\S5, some notions of skew left brace theory is interpreted via Rota---Baxter groups.

In~\S6, we consider the Yang---Baxter equation and solutions to it concerned
skew left braces and Rota---Baxter groups.

In~\S7, we introduce skew left multibraces motivated by the fact that
given a Rota---Baxter group $(G,\cdot,B)$, we get a new Rota---Baxter group
structure~$(G,\circ,B)$ on the same set~$G$.

\section{Preliminaries}

Let us recall the definition of Rota---Baxter operator on an algebra.

Let $A$ be an algebra over a field~$\Bbbk$. A linear operator $R$ on $A$ is called
a Rota---Baxter operator of weight~$\lambda\in \Bbbk$ if
\begin{equation*}\label{RBAlgebra}
R(x)R(y) = R( R(x)y + xR(y) + \lambda xy )
\end{equation*}
for all $x,y\in A$.
An algebra endowed with a Rota---Baxter operator
is called a~Rota---Baxter algebra (see, for example, \cite{GuoMonograph}).

Let us consider an analogue of Rota---Baxter operator of weight~$\pm1$ on a group.

\subsection{Rota---Baxter operators on groups}
\begin{definition}[\cite{Guo2020}]
Let $G$ be a group.

a) A map $B\colon G\to G$ is called a {\it Rota---Baxter operator} of weight~1 if
\begin{equation}\label{RB}
B(g)B(h) = B( g B(g) h B(g)^{-1} )
\end{equation}
for all $g,h\in G$.

b)  A map $C\colon G\to G$ is called a {\it Rota---Baxter operator} of weight~$-1$ if
$$
C(g) C(h) = C( C(g) h C(g)^{-1} g )
$$
for all $g,h\in G$.
\end{definition}

There is a~bijection between Rota---Baxter operators
of weights~1 and $-1$ on a~group~$G$. We will call Rota---Baxter
operators of weight~1 simply Rota---Baxter operators (RB-operators).
A group endowed with a Rota---Baxter operator is called
a~{\it Rota---Baxter group} (RB-group).

\begin{remark}
Note that the identity~\eqref{RB} has already appeared in the context of skew left braces but in quite difference sense.
In~\cite{CCD} (see also~\cite[Lemma\,1.1.17]{Bach16}), 
the authors defined so called gamma function on a finite group~$G$ as the map $\gamma\colon G\to \Aut(G)$ satisfying the equality
$\gamma(g^{\gamma(h)}h) = \gamma(g)\gamma(h)$
for all $g,h\in G$.
The difference is that an RB-operator acts from $G$ into~$G$ but not in $\Aut(G)$, and this point is important.
\end{remark}

A group $G$ is called factorizable if $G = HL$ for some its subgroups $H$ and $L$.
The expression $G = HL$ is called a factorization of $G$.
If additionally $H\cap L = \{e\}$, then such factorization is called exact.

\begin{example}[\cite{Guo2020}]\label{RBexamples}
Let $G$ be a group. Then

a) the map $B_0(g) = e$ is an RB-operator on $G$,

b) the map $B_{-1}(g) = g^{-1}$ is an RB-operator on $G$,

c) given an exact factorization $G = HL$,
a map $B\colon G\to G$ defined as follows,
$B(hl) = l^{-1}$ is a Rota---Baxter operator on~$G$.
\end{example}

We call 
the RB-operator arisen from an exact factorization (Example~\ref{RBexamples}c) as {\it splitting RB-operator}.

The following result can be useful in attempt to construct RB-operators $B$ on a given group.

\begin{lemma}[\cite{Guo2020}] \label{lem:tilda}
Let $B$ be a Rota---Baxter operator on a group~$G$.
Then $\widetilde{B}(g) = g^{-1}B(g^{-1})$ is also a Rota---Baxter operator on a group~$G$.
\end{lemma}

In particular, $\widetilde{B_0} = B_{-1}$. 
We have the equality $\widetilde{\widetilde{B}} = B$, since
$$
\widetilde{\widetilde{B}}(g)
 = g^{-1}{\widetilde{B}}(g^{-1})
 = g^{-1}(g^{-1})^{-1}B((g^{-1})^{-1})
 = B(g).
$$

The following observation is very important when one is interested on the classification of
all RB-operators on a given group.

\begin{lemma}[\cite{BG}] \label{lem:Aut}
Let $B$ be a Rota---Baxter operator on a group~$G$.
Let $\varphi$ be an automorphism of $G$. Then $B^{(\varphi)} = \varphi^{-1}B\varphi$
is a Rota---Baxter operator on a group~$G$.
\end{lemma}

Given a RB-group $(G,B)$, we may define a new binary operation $\circ \colon G\to G$~\cite{Guo2020}.

\begin{proposition}[\cite{Guo2020}]\label{prop:Derived}
Let $(G,B)$ be a Rota---Baxter group.

a) The pair $(G,\circ)$, with the multiplication
\begin{equation}\label{R-product}
g\circ h = gB(g)hB(g)^{-1},
\end{equation}
where $g,h\in G$, is also a group.

b) The operator $B$ is a Rota---Baxter operator on the group $(G, \circ)$.

c) The map $B\colon (G, \circ )\to (G, \cdot)$
is a homomorphism of Rota---Baxter groups.
\end{proposition}

\subsection{Skew left braces}
Skew left braces were introduced in \cite{GV2017} for studying non-involutive
set-theoretic solutions to the Yang---Baxter equation.
In this section we recall definitions and some known facts (see~\cite{GV2017,SV}).

A bigroupoid $(A,\cdot,\circ)$ is a set with two binary algebraic operations.
A bigroupoid $(A,\cdot,\circ)$ is called a~{\it skew left brace},
if $A^{(\cdot)} := (A,\cdot)$ and
$A^{(\circ)} := (A,\circ)$ are groups which we will call additive and multiplicative,
respectively, and
\begin{equation}\label{LSbrace}
a \circ (b \cdot  c) =  (a \circ b) \cdot a^{-1} \cdot (a \circ c)
\end{equation}
for all $a,b,c\in A$, where $a^{-1}$ denotes the additive inverse of~$a$.
For simplicity we will denote a~skew left brace $(A,\cdot,\circ)$ by~$A$.

A bigroupoid $(A,\cdot,\circ)$ is called a~{\it skew right brace},
if $A^{(\cdot)} = (A,\cdot)$ and $A^{(\circ)} = (A,\circ)$ are groups, and
\begin{equation}\label{RightBrace}
(b \cdot c) \circ a  =  (b \circ a) \cdot a^{-1} \cdot (c\circ a)
\end{equation}
holds for all $a,b,c\in A$.

If $(A,\cdot,\circ)$ is both skew left and right brace, 
then we say that $A$ is a \emph{two-sided skew brace}.

A skew left brace~$A$ is said to be a~\emph{brace} if $A^{(\cdot)}$ is an abelian group.

Given a~skew left brace $(G, \cdot, \circ)$, we call it trivial, if
either $\cdot = \circ$ or $x\cdot y = y\circ x$ for all $x,y\in G$.

Let $(G, \cdot, \circ)$ be a~skew left brace.
Then, as it is proved in~\cite{GV2017}, the map
$$
\lambda \colon (G, \circ) \to \Aut (G, \cdot),\quad a \mapsto \lambda_a,
$$
where $\lambda_a(b) = a^{-1} (a \circ b)$, is a~group homomorphism.
The inverse $a^{\circ(-1)}$ of $a \in G$ with respect to $\circ$ is given by
$\lambda_a^{-1} (a^{-1})$.
The map $\lambda_a \colon G \to G$ is bijective with inverse
\begin{equation} \label{lambda-inverse}
\lambda_a^{-1} \colon G \to G,\quad b \mapsto a^{\circ(-1)} \circ (ab).
\end{equation}
It follows that
$a \circ b = a \lambda_a(b)$, $ab = a \circ  \lambda_a^{-1}(b)$.
Hence, any skew left brace defines a~group homomorphism~$\lambda$.
Conversely, if we have a~group homomorphism~$\lambda$, then we can construct
a~skew left brace, see~\cite[Proposition~1.9]{GV2017}.

Let $G$ be a group. The holomorph of $G$ is the group
$\Hol(G) := \Aut(G)\ltimes G$, in which the product is given by
$$
(f,a)(g,b)=(fg,af(b))
$$
for all $a,b \in G$ and $f,g \in \Aut(G)$.
Any subgroup $H$ of $\Hol(G)$ acts on $G$ as follows
$$
(f,a) \cdot b = af(b),\quad a,b\in G,\ f \in \Aut (G).
$$
A~subgroup~$H$ of $\Hol(G)$ is said to be {\it regular} if for each $a \in G$
there exists a unique $(f,x) \in H$ such that $xf(a)=e$.
It is equivalent to the fact that the action of~$H$ on~$G$ is free and transitive.
Let $\pi_2\colon \Hol(G)\to G$ denote the projection map
from $\Hol(G)$ onto the second component~$G$.

The following theorem from~\cite{GV2017} provides a~connection
between skew left braces and regular subgroups.

\begin{theorem}[\cite{GV2017}]\label{gv2017}
Let $(G, \cdot, \circ)$ be a skew left brace. Then
$\{(\lambda_a, a) \mid a \in G\}$ is a regular subgroup of $\Hol(G)$,
where $\lambda_a(b) = a^{-1}(a \circ b)$ for all $b\in A$.

Conversely, if $(G, \cdot)$ is a group and~$H$ is a~regular subgroup of
$\Hol(G^{(\cdot)})$, then $(G, \cdot, \circ)$ is a skew left brace such that
$(G, \circ) \cong H$, where $a \circ b = a f(b)$
with $(\pi_2|_H)^{-1}(a)=(f, a)\in H$.
\end{theorem}

\section{Rota---Baxter operators and skew left braces}

The next statement gives a connection between RB-groups and skew left braces.

\begin{proposition}\label{RBToSLB}
Let $(G,\cdot)$ be a group and $B \colon G \to G$ be a Rota---Baxter operator. Put
$x \circ_B y = xB(x)yB(x)^{-1}$. Then $(G, \cdot, \circ_B)$ is a skew left brace.
\end{proposition}

\begin{proof}
By~Proposition~\ref{prop:Derived}, $(G,\circ_B)$ is a group.
It remains to check~\eqref{LSbrace},
$$
(g\circ_B h)g^{-1}(g\circ_B f)
 = gB(g)h\underline{B(g)^{-1}g^{-1}gB(g)}fB(g)^{-1}
 = gB(g)hfB(g)^{-1}
 = g\circ_B(hf). \qedhere
$$
\end{proof}

\begin{corollary}
Let $(G,\cdot,B)$ be a~Rota---Baxter group. If $\Imm B\subset Z(G)$,
then the skew left brace $(G, \cdot, \circ_B)$ is the trivial one.
\end{corollary}

We denote the skew left brace $(G, \cdot, \circ_B)$ by $G(B)$.
Using the operators $B_0$ and $B_{-1}$
we can construct the following skew left braces:

\begin{example}
a) $G(B_0)$ is a trivial skew left brace,

b) $G(B_{-1})$ is also a trivial skew left brace, since
$x \circ y = y \cdot x$ for all $x, y \in G$.
\end{example}

For RB-operators of weight $-1$ we can prove the following analogue of Proposition~\ref{RBToSLB}.

\begin{proposition}
Let $(G, \cdot)$ be a group and $C \colon G \to G$
be a~Rota---Baxter operator of weight $-1$.
Put $x \circ_C y = C(x)yC(x)^{-1}x$.
Then $(G, \cdot, \circ_C)$ is a~skew left brace.
If $(G, \cdot)$ is an abelian group, then
$(G, \cdot, \circ_C)$ is a trivial skew left brace.
\end{proposition}

Let $G(B)$ be a skew left brace defined by an RB-operator~$B$ on $(G,\cdot)$,
then the lambda-map is defined by the rule
$\lambda_a(b) = B(a)bB(a)^{-1}$,
i.\,e., it is an inner automorphism of~$G^{(\cdot)}$.

\begin{proposition}
Given a skew left brace $G(B) = (G, \cdot, \circ_B)$,
the group $(G, \circ_B)$ is abelian
if and only if the following identity
\begin{equation} \label{RBlsbAbelianCirc}
[y, B(x)^{-1}] [B(y)^{-1}, x] = [y, x]
\end{equation}
holds in the Rota---Baxter group $(G,\cdot,B)$ for all $x, y \in G$.
\end{proposition}

\begin{proof}
By the definition of $\circ_B$, the identity
$x\circ_B y = y\circ_B x$ is equivalent to the identity
$$
x B(x) y B(x)^{-1} = y B(y) x B(y)^{-1}.
$$
Rewrite it in the form
$$
x y B(x)^y B(x)^{-1} = y x B(y)^x B(y)^{-1}.
$$
Since, $yx = xy[y, x]$,  we get
$$
[y, B(x)^{-1}] = [y, x] [x, B(y)^{-1}]
$$
which is equivalent to~\eqref{RBlsbAbelianCirc}.
\end{proof}

We state a group analogue of the general result~\cite{Embedding}
holding for postalgebras and Rota---Baxter algebras
of an arbitrary variety.

\begin{theorem} \label{Embedding}
Every skew left brace can be embedded into a Rota---Baxter group.
\end{theorem}

\begin{proof}
Let $(G,\cdot,\circ)$ be a skew left brace.
Consider the semi-direct product $\widetilde{G} = G\ltimes G$ with the operation
\begin{equation}\label{CircProduct}
(x,y)*(z,t)
 = (x\circ z,y\lambda_x(t)).
\end{equation}

Consider the subgroups
$H = \{(g,g)\mid g\in G\}$ and $L = \{(g,e)\mid g\in G\}$
in $\widetilde{G}$.
We may decompose $\widetilde{G} = H*L$, since
$(x,y) = (y,y)*(y^{\circ(-1)}\circ x,e)$.
Define a splitting RB-operator~$B$ on $g = h*l\in\widetilde{G}$ as
$B(h*l) = l^{*(-1)}$. So, $B((x,y)) = (x^{\circ(-1)}\circ y,e)$.
Thus, $\widetilde{G}$~is an RB-group.

Embed $G$ into $\widetilde{G}$ as follows, $\psi\colon g\to (e,g)$.
Let us check that $\psi$ is an isomorphism of skew left braces
$G$ and $\Imm(\psi)$, where $\Imm(\psi)$ is considered
as a~subbrace of
$\widetilde{G}(B) = (\widetilde{G},*,\circ_B)$. Indeed,
$$
\psi(g)*\psi(h)
 = (e,g)*(e,h)
 = (e,g\cdot h)
 = \psi(g\cdot h),
$$

\vspace{-0.95cm}

\begin{multline*}
\psi(g) \circ_B \psi(h)
 = (e,g)*B((e,g))*(e,h)*(B((e,g)))^{-1} \\
 = (e,g)*(g,e)*(e,h)*(g^{\circ(-1)},e)
 = (g,g)*(g^{\circ(-1)},h)
 = (e,g\lambda_g(h))
 = (e,g\circ h)
 = \psi(g\circ h). \qedhere
\end{multline*}
\end{proof}

The following example gives some illustration of this theorem.

\begin{example}
Let $G = (\mathbb{Z}, +, \circ)$ be the left brace, where $+$
is the addition on the set of integers and
$a \circ b = a + (-1)^a b$, $a, b \in \mathbb{Z}$.
It is easy to see that inverse element under~$\circ$
is defined by the rule $a^{\circ(-1)} = (-1)^{a+1} a$.
Let us construct the RB-group containing~$G$.
Define the group operation on the set
$\widetilde{G} = \mathbb{Z} \times \mathbb{Z}$:
$$
(a, b) * (c, d) = (a+(-1)^a c, b +(-1)^a d),\quad
(a, b),(c, d) \in \widetilde{G}.
$$
One can see that the inverse element in~$\widetilde{G}$ equals
$$
(a, 0)^{*(-1)} = ((-1)^{a+1} a, 0) = (a^{\circ(-1)}, 0),\quad a \in \mathbb{Z}.
$$
Consider the subgroups
$H = \{(g,g)\mid g\in \mathbb{Z}\}$ and
$L = \{(g,0)\mid g\in \mathbb{Z}\}$ in $\widetilde{G}$.
We decompose $\widetilde{G} = H*L$, since
$$
(x,y)
 = (y,y)*(y^{\circ(-1)}\circ x,0)
 = (y, y)*((-1)^{y+1}(y-x), 0).
$$
We can define an RB-operator
$B\colon \widetilde{G} \to \widetilde{G}$ by
$$
B((g, h))
 = ((-1)^{h+1}(h-g),0)^{*(-1)}
 = ((-1)^{g+1}(h-g),0),\quad (g,h)\in G.
$$
Hence, $G$ embeds into the RB-group $(\widetilde{G}, *,B)$ by $g\mapsto(0,g)$.
\end{example}

\begin{remark} \label{tilde{G}Nilpotent}
The construction of~$\widetilde{G}$ has already appeared
(see~\cite[\S4]{NilpotentBraces} and~\cite{BR}) in the context of nilpotent skew left braces. More detailed, one can define the new operation~$\star$ on a~skew left brace~$G$,
$g\star h = g^{-1}(g\circ h)h^{-1}$.
Note that in terms of the group~$\widetilde{G}$, we have
$$
[(e,h),(g,e)]
 = (e,g\star h).
$$
This observation helps to the authors of~\cite{NilpotentBraces}
to clarify the construction of left and right series of~$G$ and hence,
the definition of left (right) nilpotency of~$G$. 
We will concern them in~\S5. 
Given an RB-group $(G,B)$, we have the equality $g\star h = [B(g)^{-1},h^{-1}]$.
\end{remark}

\begin{remark}
To prove the embedding in Theorem~\ref{Embedding}, one may take the group
$\lambda(G)\times G=\{(\lambda_g,h)\mid g,h\in G\}$
instead of~$\widetilde{G}$. Thus, we get a~subgroup of $\Hol(G^{(\cdot)})$,
and $H = \{(g,g)\mid g\in G\}$~is its regular subgroup.
However, further we will apply exactly the group $\widetilde{G}$.
\end{remark}

Now, we reprove Problem 19.90(d) from Kourovka Notebook~\cite{Kourovka}.

\begin{corollary}[\cite{TsangChao}] \label{coro:Kegel}
Let $(G,\cdot,\circ)$ be a finite skew left brace.
If $G^{(\circ)}$ is nilpotent, then $G^{(\cdot)}$ is solvable.
\end{corollary}

\begin{proof}
Given a finite skew left brace~$G$,
we may construct $\widetilde{G}$ as in the proof of Theorem~\ref{Embedding}.
We have also $\widetilde{G} = H*L$, where
$H = \{(g,g)\mid g\in G\}$ and $L = \{(g,e)\mid g\in G\}$.
Note that $H\cong L\cong G^{(\circ)}$.
Since $G^{(\circ)}$ is nilpotent, $\widetilde{G}$
is solvable by the Kegel's theorem~\cite{Kegel}.
Thus, its subgroup
$M = \{(e,g)\mid g\in G\}\cong G^{(\cdot)}$ is also solvable.
\end{proof}

\begin{corollary}[\cite{Nasybullov,TsangChao}]
Let $G$ be a skew left brace.
If $G^{(\circ)}$ is abelian, then $G^{(\cdot)}$ is metabelian.
\end{corollary}

\begin{proof}
We may repeat the proof of Corollary~\ref{coro:Kegel} to get
the factorization $\widetilde{G} = H*L$.
By the Ito's theorem~\cite{Ito}, $\widetilde{G}$ is metabelian.
Thus, its subgroup
$M = \{(e,g)\mid g\in G\}\cong G^{(\cdot)}$ is also metabelian.
\end{proof}

A group $G$ with trivial center and trivial group $\Out(G)$
of outer automorphisms is called {\it complete}.

\begin{proposition}
Let $(G,\cdot,\circ)$ be a skew left brace such that $G^{(\cdot)}$ is complete.
Then there exists a Rota---Baxter operator~$B$ on $G^{(\cdot)}$ such that $G = G(B)$.
\end{proposition}

\begin{proof}
Let us define $B$ in such a way that
\begin{equation}\label{CompleteRB}
g\circ h = gB(g)hB(g)^{-1},
\end{equation}
i.\,e., $\lambda_g(h) = h^{B(g)^{-1}}$.
Since $\Out(G)$ is trivial, we may find an element $x\in G$ such that
$\lambda_g(h) = h^x$. The element $x$ satisfying this equality is unique,
since $Z(G)$ is trivial. Let us define $B(g) = x^{-1}$.

It remains to check that $B$ is an RB-operator on $G^{(\cdot)}$.
By~\eqref{CompleteRB}, we have to show that
$B(g)B(h) = B(g\circ h)$. Since $Z(G) = \{e\}$, it is equivalent to verify
$$
s^{B(h)^{-1}B(g)^{-1}}
 = s^{B(g\circ h)^{-1}}
$$
for all $s\in G$. In terms of~$\lambda$ we check that
$$
g^{-1}(g\circ(h^{-1}(h\circ s)))
 = \lambda_g(\lambda_h(s))
 = \lambda_{g\circ h}(s)
 = (g\circ h)^{-1}(g\circ h\circ s).
$$
To confirm this we apply~\eqref{LSbrace} and its consequence
$a^{-1}(a\circ c) = a\circ (a^{\circ(-1)}\cdot c)$,
$$
g^{-1}(g\circ(h^{-1}(h\circ s)))
 = g\circ (g^{\circ(-1)}h^{-1}(h\circ s))
 = g\circ h\circ ((h^{\circ(-1)}\circ g^{\circ(-1)})s)
 = (g\circ h)^{-1}(g\circ h\circ s). \qedhere
$$
\end{proof}

\begin{remark}
Suppose that $(G,\cdot,\circ)$ is a skew left brace
such that both $G^{(\cdot)}$ and $G^{(\circ)}$ are Lie groups.
We want to clarify what an algebraic structure we get after differentiation of~$G$ in the sense which we clarify below.

By Theorem~\ref{Embedding}, we may embed~$G$
into $\widetilde{G}(B)$, where~$B$ is smooth by the definition.
Further, we differentiate the Lie RB-group $(\widetilde{G},B)$
and get the Lie RB-algebra $(L(\widetilde{G}),P)$~\cite{Guo2020}.

To make the last step, we need a definition.
A space~$A$ with two bilinear operations $[,]$ and $\cdot$
is called a {\it post-Lie algebra}~\cite{Burde_41,Vallette2007}
if $[,]$ is a Lie bracket and the next two identities hold
\begin{gather*}
x\cdot [y,z] = [x \cdot y,z] + [y,x\cdot z], \\
(x \cdot y) \cdot z - x \cdot (y \cdot z)
- (y \cdot x) \cdot z + y \cdot (x \cdot z) = [y,x]\cdot z.
\end{gather*}

It is known that every Lie algebra~$(L,[,])$ with an RB-operator~$R$
of weight~1 is a~post-Lie algebra under the products $[,]$ and
$x\cdot y = [R(x),y]$~\cite{Lax}.
Thus, differentiation of $G$ inside $(\widetilde{G},B)$
gives a~post-Lie algebra~$(L(G),[,],\cdot)$.
\end{remark}

In~\cite[Problem~12]{Problems}, it was stated the question how looks like
a~free skew left brace. In~\cite{Orza}, some implicit
construction of the free skew left brace was suggested.
Let us show that the construction of a~free Rota---Baxter group leads to the knowledge
of how a~free skew left brace looks like. Recall that a~free Rota---Baxter group
is a~free algebraic structure in the signature~$\langle\cdot,{}^{-1},B\rangle$
with one binary and two unary operations.

\begin{proposition}
Let $(G,B)$ be a free RB-group generated by a~set~$X$.
Denote by $S$ the skew left subbrace of $G(B)$ generated by~$X$.
Then $S$~is a~free skew left brace generated by~$X$.
\end{proposition}

\begin{proof}
Let $(A,\cdot,\circ)$ be a~skew left brace generated by~$X$.
By Theorem~\ref{Embedding}, $A$~embeds into the RB-group~$\tilde{A}$,
which itself is generated by the set~$X$ as the Rota---Baxter group.
Thus, there exists a surjective homomorphism $\psi\colon G \to \tilde{A}$ of RB-groups.
Hence, $\psi$~is a~surjective homomorphism from $S$ to the~$A$
considered as a~subbrace of~$\tilde{A}$.
Since $A$ is an arbitrary skew left brace generated by~$X$,
$S$ is generated by~$X$ too, and $A$ is a homomorphic image of~$S$,
we conclude that $S$~is the free skew left brace generated by~$X$.
\end{proof}

\section{Constructions of skew left braces via RB-groups}

Let us apply Lemmas~\ref{lem:tilda} and~\ref{lem:Aut} to get RB-groups from the given one.
Let $(G,\cdot,B)$ be an RB-group, then
$(G,\circ_B)\cong (G,\circ_{\widetilde{B}})\cong (G,\circ_{B^{(\varphi)}})$,
where $\varphi\in\Aut(G^{(\cdot)})$. Indeed,
$$
g\circ_{\widetilde{B}}h
 = g(g^{-1}B(g^{-1}))h(g^{-1}B(g^{-1}))^{-1}
 = B(g^{-1})h B(g^{-1})^{-1}g
 = (g^{-1}\circ_B h^{-1})^{-1},
$$
so, ${}^{-1}\colon x\to x^{-1}$ is an isomorphism of
the groups $(G,\circ_B)$ and $(G,\circ_{\widetilde{B}})$.
Similarly,
\begin{multline*}
g\circ_{B^{(\varphi)}}h
 = g\varphi^{-1}(B(\varphi(g)))h(\varphi^{-1}(B(\varphi(g))))^{-1}
 = \varphi^{-1}( \varphi(g)B(\varphi(g))\varphi(h)B(\varphi(g))^{-1}) \\
 = \varphi^{-1}(\varphi(g)\circ_B\varphi(h)),
\end{multline*}
thus, $\varphi(g\circ_{B^{(\varphi)}}h) = \varphi(g)\circ_B \varphi(h)$,
and the corresponding groups are isomorphic.
In this case, $\varphi$ is an isomorphism of skew left braces
$(G,,\cdot,\circ_B)$ and $(G,\cdot,\circ_{B^{(\varphi)}})$.

By every splitting RB-operator~$B$ defined by an exact factorization~$G = HL$
as $B(hl) = l^{-1}$, we get the skew left brace $G(B)$ with the product
$$
x \circ y
 = (hl)\circ y
 = hlB(hl)yB(hl)^{-1}
 = hyl.
$$
Such skew left braces appeared several times, see~\cite[Example~1.6]{GV2017},
~\cite[Theorem~2.3]{SV}, and~\cite[p.~19]{BNY-1}.

\begin{proposition}[\cite{BG}] \label{triangular}
Let $G$ be a group such that $G = HLM$, where $H$, $L$, and $M$
are subgroups of $G$ with pairwise trivial intersection.
Let $C$ be a Rota---Baxter operator on~$L$.
Moreover, $[H,L] = [C(L),M] = e$.
Then the map $B\colon G\to G$ defined by the formula
$$
B(hlm) = C(l)m^{-1}
$$
is a Rota---Baxter operator on~$G$.
\end{proposition}

We have $G^{(\circ)}\cong H\times L_C\times M^{\op}$, where
$M^{\op}$ is the set $M$ with the opposite product $m*m' = m'm$.
Indeed,
\begin{multline*}
x\circ y
 = (hlm)\circ (h'l'm')
 = (hlm)B(hlm)h'l'm'B(hlm)^{-1}
 = hlmC(l)m^{-1}h'l'm'mC(l)^{-1} \\
 = hl\underline{mm^{-1}}C(l)h'l'm'mC(l)^{-1}
 = hh'lC(l)l'm'mC(l)^{-1}
 = hh'(l\circ_L l') m'm.
\end{multline*}

\begin{proposition}[\cite{BG}]
Given a semidirect product $G = H\rtimes L$, let $C$ be a~Rota---Baxter operator on $L$.
Then a map $B\colon G\to G$ defined by the formula $B(hl) = C(l)$, where $h\in H$ and $l\in L$,
is a~Rota---Baxter operator.
\end{proposition}

Thus, we have the circle product equal to
\begin{multline*}
x\circ y
 = (hl)\circ (h'l')
 = hlB(hl)h'l'B(hl)^{-1}
 = hlC(l)h'l'C(l)^{-1}
 = hh'^{C(l)^{-1}l^{-1}}lC(l)l'C(l)^{-1} \\
 = hh'^{C(l)^{-1}l^{-1}}(l\circ_C l'),
\end{multline*}
and $G^{(\circ)}\cong H\rtimes L_C$.

\begin{proposition}[\cite{BG}] \label{prop:RB-Hom}
a) Let $(G, B)$ be a Rota---Baxter group and $B$ be an automorphism of $G$. Then $G$ is abelian.

b) If $G$ is a group and $H$ is its abelian subgroup,
then any homomorphism (or antihomomorphism) $B\colon G \to H$ is a Rota---Baxter operator.
\end{proposition}

Given a group~$G$ and its abelian subgroup~$H$ a homomorphism from~$G$ to~$H$ defines
an RB-operator on~$G$ by Proposition~\ref{prop:RB-Hom}.
In this case, we can not derive more information than given in the definition
formula for the product~$\circ$.

Let $R$ be a~Rota---Baxter operator of weight~1 on
the associative algebra equal to the direct sum of fields
$\Bbbk^n = \Bbbk e_1\oplus \Bbbk e_2\oplus \ldots \oplus \Bbbk e_n$, where $e_ie_j = \delta_{ij}e_i$.
A~linear operator $R(e_i) = \sum\limits_{k=1}^n r_{ik}e_k$,
$r_{ik}\in \Bbbk$, is an RB-operator of weight~1 on~$\Bbbk^n$
if and only if the following conditions are satisfied~\cite{AnBai,Braga}:

(1) $r_{ii} = 0$ and $r_{ik}\in\{0,1\}$
or $r_{ii} = -1$ and $r_{ik}\in\{0,-1\}$ for all $k\neq i$;

(2) if $r_{ik} = r_{ki} = 0$ for $i\neq k$,
then $r_{il}r_{kl} = 0$ for all $l\not\in\{i,k\}$;

(3) if $r_{ik}\neq0$ for $i\neq k$,
then $r_{ki} = 0$ and
$r_{kl} = 0$ or $r_{il} = r_{ik}$ for all $l\not\in\{i,k\}$.

Moreover, we may suppose that the matrix of $R$ is upper-triangular~\cite{Braga,SumOfFields}.

\begin{proposition}[\cite{BG}] \label{theo:directProduct}
Let $G^n = G\times G\times \ldots \times G$ and let $\Bbbk$ be a~field.
Let $R$ be a~Rota---Baxter operator of weight~1 on
$\Bbbk e_1\oplus \Bbbk e_2\oplus \ldots \oplus \Bbbk e_n$,
$R(e_i) = \sum\limits_{i=1}^n r_{ik}e_k$, and the matrix of $R$ is upper-triangular.
Let $\psi_2,\ldots,\psi_n\in\Aut(G)$. Then a~map $B\colon G^n\to G^n$
defined by the formula
\begin{equation} \label{RBOnDirectProduct}
B((g_1,\ldots,g_n)) = (t_1,\ldots,t_n),\quad
t_i = g_i^{r_{ii}}\psi_i\big(g_{i-1}^{r_{i-1\,i}}\psi_{i-1}\big(g_{i-2}^{r_{i-2\,i}}
 \ldots \psi_2\big(g_1^{r_{1i}}\big)\big)\big),\ i\geq2,
\end{equation}
is a Rota---Baxter operator on $G^n$.
\end{proposition}

Analogously to~\cite{SumOfFields}, it is not difficult to state
that $G^n_B\cong G^n$.

Note that C. Tsang recently studied~\cite{Tsang2}
skew left braces $(A,\cdot,\circ)$ with $A^{(\cdot)} = G^n$
and $A^{(\circ)} \cong G^n$, where $G$ is non-abelian simple group.
Earlier~\cite{Tsang}, he showed that
there are exactly
$$
e(G,G)
 = 2^n(n|\Aut(G)|+1)^{n-1}
$$
regular subgroups of $\Hol(G^n)$ which are isomorphic to $G^n$.
It is interesting to compare this value with $2^n(n+1)^{n-1}$,
the number of all RB-operators of weight~1 on the algebra~$\Bbbk^n$~\cite{SumOfFields}.

In~\cite{Tsang2}, C. Tsang studied the number $b(G,n)$ of isomorphism classes
of skew left braces $(A,\cdot,\circ)$ such that
$A^{(\cdot)} \cong A^{(\circ)} \cong G^n$, where $G$ is a~finite non-abelian simple group.
It turns out that $b(G,n)$ does not depend on the choice of $G$ and so, it is a function on~$n$.
The first values $1,7,26,107,458,2058,\ldots$ of $b(G,n)$ computed in~\cite{Tsang2} with the help of Magma
coincide (except the first one) with the sequence A000151 from OEIS~\cite{OEIS},
which was shown in~\cite{SumOfFields} to be equal to the number of orbits
of RB-operators of weight~1 on the algebra~$\Bbbk^n$
under the action of the group $\Aut(\Bbbk^n)\cong S_n$.

Thus, we have the following natural questions.

\begin{question}
Let $G$ be a~finite non-abelian simple group.

a) Do RB-operators on $G^n$ arisen by~\eqref{RBOnDirectProduct}
from RB-operators of weight~1 on the algebra~$\Bbbk^n$ lying in different orbits
produce non-isomorphic skew left braces?

b) Do all skew left braces $(A,\cdot,\circ)$ such that
$A^{(\cdot)} \cong A^{(\circ)} \cong G^n$
are defined via RB-operators on $G^n$ by~\eqref{RBOnDirectProduct}?
\end{question}

\section{Skew left braces theory in terms of RB-operators}\label{TermsViaRB}

\subsection{Skew left braces and 1-cocycles}
Let $G$ and $\Gamma$ be groups and assume that
$\Gamma \times G \to G$,
$(\gamma, a) \mapsto \gamma \cdot a$,
is left action of $\Gamma$ on $G$ by automorphism.
A bijective {\it 1-cocycle} (or {\it crossed homomorphism})
is a~bijective map
$\pi \colon \Gamma \to G$ such that
$\pi(\gamma \delta)
 = \pi(\gamma)(\gamma \cdot \pi(\delta))$
for all $\gamma,\delta\in \Gamma$.

There is a connection between bijective 1-cocycles and
skew left braces~\cite[Proposition 1.11]{GV2017}).
If $G^{(\cdot)}$ is a~group and $\pi \colon \Gamma \to G$
is a bijective 1-cocycle, then for the operation~$\circ$ defined on $G$ as follows,
\begin{equation} \label{1-cocycle}
a \circ b = \pi(\pi^{-1}(a) \pi^{-1}(b)),\quad a,b \in G,
\end{equation}
we have that $(G, \cdot, \circ)$ is a~skew left brace.

Conversely, assume that $G$ is a~skew left brace.
Set $\Gamma = G$ with multiplication
$(a,b)\mapsto a\circ b$ and $\pi = \id$.
Then $a \mapsto \lambda_a$ is a group homomorphism
and hence $\Gamma$ acts on $G$ by automorphisms and
hence, $\pi \colon \Gamma \to G$ is a bijective 1-cocycle.

In particular case, when $G$ acts on $\Gamma = G$ by conjugation,
the inverse of a bijective 1-cocycle by~\eqref{1-cocycle}
is nothing more than an RB-operator on~$G$~\cite{Guo2020}.

\begin{proposition}
Let $(G,\cdot)$ be a group and $B \colon G \to G$ be a Rota---Baxter operator. Put
$x \circ_B y = xB(x)yB(x)^{-1}$. Then $(G,\cdot,\circ_B)$ is a two-sided skew brace
if and only if for every $g\in G$ the map
$\psi_g(x)\colon x\to [B(x)^{-1},g] = x\star g^{-1}$
is a~1-cocycle acting from~$G^{\op}$ to~$G^{\op}$.
\end{proposition}

\begin{proof}
By~Proposition~\ref{RBToSLB}, $G$ is a skew left brace.
Rewrite~\eqref{RightBrace} in terms of $B$,
$$
abB(ab)cB(ab)^{-1}
 = aB(a)cB(a)^{-1}c^{-1}bB(b)cB(b)^{-1},
$$
which is equivalent to the equality
$$
b[B(ab)^{-1},c^{-1}]
 = [B(a)^{-1},c^{-1}]b[B(b)^{-1},c^{-1}].
$$
Thus, $\psi_{c^{-1}}(ab) = \psi_{c^{-1}}(a)^b \psi_{c^{-1}}(b)$ and so,
$\psi_g(b*a) = \psi_g(b)*(b\cdot \psi_g(a))$ for every $a,b,g\in G^{\op}$.
Here $b*a = ab$ is the opposite product in~$G$.
\end{proof}

\subsection{Ideals}
A~non-empty subset $I$ of a~skew left brace $(G,\cdot,\circ)$ is said to be an {\it ideal} of~$G$~\cite{GV2017} if 

1) $\lambda_a(I) \subseteq I$ for all $a \in G$,

2) $I$ is normal in $(G, \cdot)$,

3) $I$ is normal in~$(G, \circ)$. 

For example, the kernel of any skew left brace homomorphism is an ideal.
The notion of ideal allows us to consider the quotient skew left brace~$G/I$. Note that the condition~1) is equivalent to the equality $a\circ I = aI$ for all $a\in G$.

A~non-empty subset $I$ of $G$ is called ({\it strong}) {\it left ideal} of~$G$ if 1) holds and $I$ is a~(normal) subgroup of $(G, \cdot)$.
Strong left ideals provide decompositions of the corresponding solutions to the Yang---Baxter equation~\cite{Jespers} and they are connected with intermediate fields of a~Galois field extension~\cite{Quasi-ideal}. Every left ideal is a skew left subbrace.

Inspired by Theorem~\ref{Embedding}, Remark~\ref{tilde{G}Nilpotent},
and the analogy with postalgebras, we apply the group~$\widetilde{G}$
and concerned map~$\psi\colon g\to (e,g)$ as the natural source of definitions of different classes of skew left braces.

\begin{definition}\label{DefViaDelta}
Let~$G$ be a~skew left brace.

a) \cite{NilpotentBraces} 
Let $\mathcal{X}$ be a class of groups.
We say that a~non-empty subset $I\subseteq G$ is of type~$\mathcal{X}$
if $I$ is of type $\mathcal{X}$ in the group~$(G,\cdot)$.

b) Let $\mathcal{X}$ be a class of groups.
We say that a~non-empty subset $I\subseteq G$ is of strong type $\mathcal{X}$
if $\psi(I)$ is of type $\mathcal{X}$ in the group~$\widetilde{G}$.
\end{definition}

Note that a lot of notions of skew left braces theory were interpreted in~\cite{BR} via the group~$\widetilde{G}$. 
The next statement was mentioned in~\cite[p.\,8]{BR} in one direction.

\begin{proposition} \label{IdealCriterion}
Let $G$ be a skew left brace.
A~non-empty subset~$I\subseteq G$ is a~strong left ideal of~$G$ if and only if
$\psi(I)$ is a normal subgroup of~$\widetilde{G}$.
\end{proposition}

\begin{proof}
Let $a\in G$ and $i\in I$.
The condition that $\psi(I)$ is a normal subgroup of~$\widetilde{G}$ is equivalent
to the following two inclusions fulfilled in~$\widetilde{G}$,
$$
(e,a)^{*(-1)}*(e,i)*(e,a)\in (e,I), \quad
(a,e)^{*(-1)}*(e,i)*(a,e)\in (e,I).
$$
By~\eqref{CircProduct}, we obtain that $a^{-1}ia\in I$
and $\lambda_{a^{\circ(-1)}}(i)\in I$.
Since~the elements $a$ and $i$ have been chosen arbitrary,
it means that $\psi(I)$ is a normal subgroup of~$\widetilde{G}$
if and only if~$I$ is a~normal subgroup of $G^{(\cdot)}$ and
$\lambda_x(I)\subseteq I$ for all $x\in G$.
It is exactly the definition of a~strong left ideal in~$G$.
\end{proof}

\begin{proposition}
Let $G(B) = (G,\cdot,\circ)$ be a skew left brace obtained by an RB-operator~$B$
on $G^{(\cdot)}$ and $I\subseteq G$. Then 

a) $I$~is an ideal of~$G(B)$ if and only if $I$~is a~normal subgroup of both
$G^{(\circ)}$ and $G^{(\cdot)}$.

b) $I$~is a left ideal of~$G(B)$ if and only if $I$ is a subgroup of $G^{(\cdot)}$ normalized by $\Imm(B)$.
\end{proposition}

\begin{proof}
a) If $I$ is an ideal of~$G$, then $I$~is a~normal subgroup
in both $G^{(\cdot)}$ and $G^{(\circ)}$
and also $\lambda_a(I) \subseteq I$ for all $a \in G$.
The last condition means that $B(a)iB(a)^{-1}\in I$
for all $a\in G$, $i \in I$. It follows from normality of~$I$ in $G^{(\cdot)}$.

b) It follows from the condition $B(a)iB(a)^{-1}\in I$ holding for all $a\in G$, $i \in I$.
\end{proof}

Let us introduce as far as we know the new notion of the left center of a~skew left brace.

\begin{definition}
Let $(G, \cdot, \circ)$ be a~skew left brace.
The {\it left center} of~$G$ is defined as follows,
$$
Z_l(G)
 = \{c\in G\mid c\in Z(G^{(\cdot)}),\,
 gc = g\circ c\mbox{ for all }g\in G)\}.
$$
\end{definition}

By the following statement and by Proposition~\ref{IdealCriterion}, 
$Z_l(G)$ is a strong left ideal in~$G$.

\begin{proposition}
a) Let $G$ be a skew left brace. Then $\psi(Z_l(G)) = Z(\widetilde{G})\cap \psi(G)$.

b) Let $G(B) = (G,\cdot,\circ)$ be a~skew left brace
obtained by an RB-operator~$B$ on $G^{(\cdot)}$. Then
$Z_l(G(B)) = Z(G^{(\cdot)})$.
\end{proposition}

\begin{proof}
a) An element~$(e,s)$ lies in $Z(\widetilde{G})$ if and only if
$(a,b)*(e,s) = (e,s)*(a,b)$ for all $a,b\in G$. It means that
$b\lambda_a(s) = s\lambda_e(b) = sb$ or equivalently,
$\lambda_a(s) = s^b$. Taking~$a = e$, we get that $s\in Z(G^{(\cdot)})$.
Thus, $s = \lambda_a(s) = a^{-1}(a\circ s)$ for all $a\in G$, i.\,e.,
$as = a\circ s$.

b) Let $c\in Z_l(G(B))$, then $c\in Z(G^{(\cdot)})$.
If $c\in Z(G^{(\cdot)})$, then $g\circ c = gB(g)cB(g)^{-1} = gc$ for all $g\in G$.
\end{proof}

In general, the left center is not an ideal of a skew left brace.

\begin{example}[I. Colazzo]
Let $A = \mathbb{Z}_3 = \langle a\rangle$ and 
$B = \mathbb{Z}_2 = \langle b\rangle$ be trivial braces. 
For a~homomorphism $\beta \colon B \to \Aut(A)$ defined as follows, $\beta(b)$ acts on~$A$ by the rule $x\to x^2$, 
we consider the semi-direct product $G = A\rtimes B$ with $(G,\cdot)\cong \mathbb{Z}_3\times\mathbb{Z}_2$ and $(G,\circ)\cong S_3$. 
It is not hard to show that $Z_l(G) = \{e\}\rtimes B$ which is not a normal subgroup of $(G,\circ)$.
\end{example}

Let $(G, \cdot, \circ)$ be a~skew left brace.
The {\it socle} of~$G$ was defined and actually reformulated in~\cite{GV2017} as follows,
$$
\Soc(G)
 = \{a\in G \mid a\in Z(G^{(\cdot)}),\,a\circ b = ab\mbox{ for all }b\in G\}.
$$

In comparison with the notion of the left center, one may call $\Soc(G)$ as the right center of~$G$.

We give the following description of socles in skew left braces
constructed by RB-operators.

\begin{proposition}
Let $G(B) = (G,\cdot,\circ_B)$ be a skew left brace obtained by an RB-operator~$B$
on $G^{(\cdot)}$. Then
$$
\Soc(G) = \{ a \in G \mid a,B(a) \in Z(G^{(\cdot)})\}
 = Z(G^{(\cdot)})\cap B^{-1}[Z(G^{(\cdot)})].
$$
\end{proposition}

\begin{proof}
The equality $a\circ b = ab$ is equivalent to $B(a)bB(a)^{-1} = b$,
i.\,e., $B(a)\in Z(G^{(\cdot)})$.
\end{proof}

In~\cite{Annulator}, the {\it annihilator} of a~skew left brace~$G$ was defined as $\Ann(G) = Z(G^{(\circ)})\cap\Soc(G)$. 
Thus, we may redefine it as $\Ann(G) = Z_l(G)\cap\Soc(G)$.
In~\cite{Bonatto}, $\Ann(G)$ was called as {\it center}, and it was applied to develop central nilpotency of skew left braces.

Define an upper central series of a skew left brace~$G$ as follows,
$\zeta_1(G) = Z_l(G)$ and $\zeta_{n+1}(G)$ is a strong left ideal of~$G$ such that
$$
\zeta_{n+1}(G)/\zeta_n(G) = \psi(G)/\zeta_n(G)\cap Z(\widetilde{G}/\zeta_n(G)).
$$
Here we identify $\zeta_i(G)$ with its image in $\widetilde{G}$ under the action of~$\psi$.

\begin{definition}
A~skew left brace $G$ is called {\it strong left nilpotent} if there exists~$m$ such that $\zeta_m(G) = G$.
\end{definition}

In~\cite{NilpotentBraces}, $\star$-left and -right nilpotency of skew left braces were defined. The left series of a~skew left brace~$G$ is the sequence
$G \subseteq G^2 \subseteq G^3 \subseteq \ldots$, where 
$G^{n+1} = G \star G^n$. Recall that $g\star h = g^{-1}(g\circ h)h^{-1}$.
A skew left brace~$G$ is said to be {\it left $\star$-nilpotent} if there is a~positive natural~$n$ such that $G^n = \{e\}$. 

\begin{proposition}
Let $G$ be a strong left nilpotent skew left brace.
Then $G$ is left $\star$-nilpotent and $G^{(\cdot)}$ is nilpotent.
\end{proposition}

\begin{proof}
Since $Z_l(\widetilde{G}/\zeta_n(G)) \subset Z(\psi(G)/\zeta_n(G))$, 
we derive that $\zeta_1(G),\zeta_2(G),\ldots$ is a central series of the group~$G^{(\cdot)}$, and hence, $G^{(\cdot)}$ is nilpotent.

Suppose that $\zeta_m(G) = G$, then by the definition of $\zeta_i(G)$, we have $G\star \zeta_i(G)\subseteq \zeta_{i+1}(G)$. Thus, $G^{m+1} = \{e\}$ and $G$ is left $\star$-nilpotent.
\end{proof}

\begin{theorem}
Let $G$ be a strong left nilpotent skew left brace and let $I$ be a nontrivial strong left ideal in~$G$.
Then $Z_l(G)\cap I\neq \{e\}$.
\end{theorem}

\begin{proof}
Suppose that $Z_l(G)\cap I = e$. Find $k\geq1$ such that $\zeta_k(G)\cap I = \{e\}$ and $\zeta_{k+1}(G)\cap I \neq \{e\}$.
Take $(e\neq)i\in \zeta_{k+1}(G)\cap I$. 
By the definition of $\zeta_{k+1}(G)$, we have $i\in Z(\widetilde{G}/\zeta_k(G))$.
Thus, $[i,g]\in \zeta_k(G)$ for all $g\in \widetilde{G}$. 
Since $I$ is a normal subgroup of~$\widetilde{G}$, we have also $[i,g]\in I$.
By the assumption, $\zeta_k(G)\cap I = \{e\}$, so, $[i,g] = e$ for all $g\in\widetilde{G}$.
Finally, we conclude that $i\in Z(\widetilde{G})$ and therefore, $i\in \zeta_1(G)\cap I$, a~contradiction.
\end{proof}

\subsection{$\lambda$-homomorphic skew left braces}

In~\cite{BNY-1}, $\lambda$-homomorphic skew left braces were introduced as the tool to construct skew left braces.

\begin{definition}[\cite{BNY-1}]
A skew left brace $(G, \cdot, \circ)$ is said to be a $\lambda$-{\it homomorphic
skew left brace} if the map $\lambda \colon (G, \cdot) \to \Aut (G,\cdot)$, defined
by $\lambda_a(b) = a^{-1}(a\circ b)$, $a,b\in G$, is a~homomorphism.
\end{definition}

For a~group $G = (G, \cdot)$, we desire to define
a~homomorphism $\lambda \colon G \to \Aut (G)$ such that
$(G, \cdot, \circ)$ is a~skew left brace, where $\circ$ is defined by
$a \circ b = a \cdot \lambda_a (b)$ for all $a, b \in G$.
The following result holds.

\begin{theorem}[\cite{BNY-1}] \label{t1}
Let~$G$ be a~group, $\lambda \colon G \to \Aut(G)$ be a~homomorphism.
The set $H_{\lambda} := \{(\lambda_a,a) \mid a \in G \}$
is a~subgroup of~$\Hol(G)$ if and only if
$$
[G, \lambda(G)]
 := \{ b^{-1} \lambda_a (b) \mid a, b \in G \} \subseteq \ker \lambda.
$$
Moreover, if  $H_{\lambda}$ is a subgroup, then it is regular, and
therefore by Theorem~\ref{gv2017} we get a~skew left brace $(G,\cdot,\circ)$,
where $\circ$ is defined by $a \circ b = a \lambda_a(b)$.
\end{theorem}

Connection between skew left braces defined by RB-groups
and $\lambda$-homomorphic skew left braces gives

\begin{proposition}
Let $G(B) = (G,\cdot,\circ_B)$ be a skew left brace obtained by an RB-operator~$B$
on $G^{(\cdot)}$. Then $G(B)$ is $\lambda$-homomorphic skew left brace
if and only if $B(ac)^{-1}B(a)B(c)\subseteq Z(G^{(\cdot)})$
for all $a,c\in G$.
\end{proposition}

\begin{proof}
We have $\lambda_{ac}(b) = B(a c) b B(ac)^{-1}$.
On the other hand,
$$
\lambda_{a} \lambda_{c}(b)
 = \lambda_{a} ( B(c) b B(c)^{-1} )
 = B(a) B(c) b B(c)^{-1} B(a)^{-1}.
$$
Hence, $\lambda$ is homomorphism if and only if
$b^{B(ac)^{-1}} = b^{(B(a)B(c))^{-1}}$ for $a,b,c \in G$.
\end{proof}

\section{Yang---Baxter equation}

A~set-theoretical solution to the (quantum) Yang---Baxter equation (YBE)
on a~set~$X$ is a~bijective map $S\colon X\times X\to X\times X$ such that
\begin{equation} \label{SYBE}
(S \times \id) (\id \times S) (S \times \id)
 = (\id \times S) (S \times \id) (\id \times S).
\end{equation}

Let us represent $S\colon X\times X\to X\times X$ as 
$S(x,y) = (\sigma_x(y), \tau_y(x))$ for some functions
$\sigma_x, \tau_y \colon X \to X$, $x, y \in X$.
A solution~$S$ of~\eqref{SYBE} is called
{\it non-degenerate} if the maps $\sigma_x$ and $\tau_x$ are bijective for each $x\in X$.
A solution~$S$ is called {\it involutive} if $S^2 = \id$.
Define the map $P\colon X\times X\to X\times X$ as follows,
$P(x,y) = (y,x)$, it is a solution to~\eqref{SYBE}.

We prove that any Rota--Baxter group gives a set-theoretic solution
to the Yang---Baxter equation.

\begin{theorem}
Let $(G, \cdot, B)$ be a Rota---Baxter group.
Put $\lambda_a(b) = B(a) b B(a)^{-1}$ for $a, b \in G$, then
\begin{equation} \label{RBE}
S \colon G \times G \to G\times G,\quad
S(a, b) = (\lambda_a(b), a^{\lambda_a(b)B(\lambda_a(b))}),
\end{equation}
is a~non-degenerate set-theoretical solution to the Yang---Baxter equation.
Moreover, $S$ is involutive if and only if $ab = ba$ for all $a, b \in G$.
\end{theorem}

\begin{proof}
In \cite[Theorem 3.1]{GV2017}, it was proved that if $(G,\cdot, \circ)$
is a~skew left brace, then
\begin{equation} \label{BraceToYBE}
S(a, b) = (\lambda_a(b), \lambda^{-1}_{\lambda_a(b)}((a \circ b)^{-1} a(a \circ b)),
\end{equation}
is a~non-degenerate set-theoretical solution to the Yang---Baxter equation.
Moreover, $S$~is involutive if and only if $ab = ba$ for all $a, b \in G$.

In the Rota---Baxter group $(G, \cdot, B)$, the map $\lambda_a$
equals the conjugation by $B(a)^{-1}$,
so $\lambda_a^{-1}$ is the conjugation by $B(a)$.
Further, by the definition
$a \circ b = a B(a) b B(a)^{-1}$.
So,
$$
(a \circ b)^{-1} a (a \circ b)
 = B(a) b^{-1} B(a)^{-1} a B(a) b B(a)^{-1}
 = a^{B(a) b B(a)^{-1}},
$$
and
\begin{multline*}
\lambda^{-1}_{\lambda_a(b)}((a \circ b)^{-1} a (a \circ b))
 = B(\lambda_a(b))^{-1}a^{B(a) b B(a)^{-1}} B(\lambda_a(b)) \\
 = a^{B(a) b B(a)^{-1} B(\lambda_a(b))}
 = a^{\lambda_a(b) B(\lambda_a(b))}. \qedhere
\end{multline*}
\end{proof}

Let $X$ be a non-empty set and $S \colon X \times X \to X \times X$
be a~non-degenerate set-theoretical solution to the Yang---Baxter equation.
Applying~\cite{GV2017}, one can construct a~skew left brace and
embed this skew left brace into a~Rota---Baxter group by Theorem~\ref{Embedding}.

\begin{question}
When can one define a Rota---Baxter group on~$X$ by~$S$
not involving skew left braces in the process?
\end{question}

Actually, set-theoretic solutions to the Yang---Baxter equation
were studied before V.G.~Drinfeld gave his examples~\cite{Drinfeld}.
D.~Joyce \cite{Joyce} and S.~Matveev~\cite{Matveev} introduced
quandles and racks as invariants of knots and links.

Recall the required definitions.

A \textit{rack} $R$ is a groupoid $(R, *)$  which satisfies the following two axioms:

(r1) the map $I_x:y\mapsto y*x$ is a bijection of $R$ for all $x\in R$,

(r2) $(x*{y})*z=(x*z)*({y*z})$ for all $x,y,z\in R$.

Axioms (r1) and (r2) imply that the map $I_x$ is an automorphism of $R$ for all $x\in R$.
The group $\Inn(R)=\langle I_x \mid x\in R\rangle$
is called the \textit{group of inner automorphisms} of $R$.
By~(r2), we have that $I_{x*y}=I_yI_xI_y^{-1}$ holds for all $x,y\in R$.

The group $\Inn(R)$ acts on~$R$ naturally. A~rack~$R$
which satisfies the additional axiom (q) $x*x=x$ for all $x\in R$
is called a \textit{quandle}. The simplest example of a~quandle
is the trivial quandle on a set $X$, that is the quandle $Q=(X,*)$,
where $x*y=x$ for all $x,y\in X$. A~lot of examples of quandles come from groups.

\begin{example}\label{conjugationex}
Let $G$~be a~group. Then the set $G$ under the product
$x*y=x^y = y^{-1}xy$ is a~quandle
called the {\it conjugation quandle} of the group $G$ and it is denoted by $\Conj(G)$.
\end{example}

A non-degenerate set-theoretic solution~$S$ to YBE, where $S(x,y)=(y,\tau_y(x))$,
is said to be of the {\it rack type} if the groupoid $(X, *)$, where $x*y = \tau_y(x)$ is a~rack.
If $(X,*)$ is a~quandle, we say that this solution is of the {\it quandle type}.
In particular, the permutation solution $P(x,y) = (y,x)$ has the quandle type.

The connection between racks and the solutions to YBE of the rack type
gives the following known result (see, for example, \cite{Joyce}).

\begin{proposition}
Let $X$ be a set. The map $S\colon X\times X\to X\times X$ defined as follows,
$S(x,y) = (y, \tau_y(x))$ is a~non-degenerate solution to YBE
if and only if $(X,*)$ is a rack, where $x*y = \tau_y(x)$.
\end{proposition}

\begin{definition}
Two non-degenerate solutions $S$ and $S'$ to YBE on~$X$
are called conjugate if there exists an invertable map
$T\colon X \times X \to X \times X$ such that $S' = TST^{-1}$.
\end{definition}

\begin{remark}
If $S(x,y) = (y, \tau_y(x))$ is a~solution of the rack type
on a~set~$X$, then $PSP(x,y) = (\tau_x(y), x)$ gives
a~non-degenerate solution on~$X$.
\end{remark}

\begin{proposition}[{\cite[Proposition 3.7]{GV2017}}] \label{GV}
Any solution to YBE constructed by a~skew left brace $(G,\cdot,\circ)$
by the formula~\eqref{BraceToYBE} is conjugate to the solution
$S'(a, b) = (b, a^b)$, $a, b \in G$, of the quandle type.
\end{proposition}

Since we have constructed the solution to YBE by an RB-group via skew left braces,
we have stated the following result too.

\begin{corollary}
The solution~\eqref{RBE} is conjugate to the solution of the quandle type, $S'(a, b) = (b, a^b)$.
\end{corollary}

In fact, A.~Soloviev~\cite[Theorem 2.3]{Solovev}
(see also L.~Guarnieri, L.~Vendramin~\cite[Proposition 3.7]{GV2017}
and D.~Bachiller \cite[Proposition 5.2]{Bach}) proved that any
non-degenerate solution is conjugate to a solution of the rack type.

A solution $S(x,y) = (\sigma_x(y), \tau_y(x))$ to YBE is called left non-degenerate if
the map $\sigma_x$ is a~bijection for every $x \in X$.

\begin{proposition}
If $S(x,y) = (\sigma_x(y), \tau_y(x))$, $x, y \in X$, gives
a~left non-degenerate solution to YBE on~$X$, then it conjugates
to a~solution of the form
$S'(x,y) = (y, \sigma_y(\tau_{\sigma^{-1}_x(y)}(x))$.
If for all $a, b \in X$ there exists a unique $x \in X$ such that
\begin{equation}\label{cond}
\tau_{\sigma^{-1}_x(a)}(x) = \sigma^{-1}_a(b),
\end{equation}
then this solution is of the rack type.
\end{proposition}

\begin{proof}
Take the map $T\colon X \times X \to X \times X$, that is defined by the rule
$T(x, y) = (x, \sigma_x(y))$, $x, y \in X$.
It has inverse $T^{-1}(x, y) = (x, \sigma^{-1}_x(y))$. Then
$$
T S T^{-1}(x, y)
 = TS(x, \sigma^{-1}_x(y))
 = T ( \sigma_x(\sigma^{-1}_x(y)),\tau_{\sigma^{-1}_x(y)}(x))
 = (y, \sigma_y(\tau_{\sigma^{-1}_x(y)}(x)).
$$
By \cite[Theorem 2.3]{Solovev}, this is a~solution and axiom~(r2) holds.
Further, this solution is of the rack type if and only if
the groupoid $(X,*)$ with the operation
$$
x * y = \sigma_y(\tau_{\sigma^{-1}_x(y)}(x))
$$
is a~rack. This operation satisfies~(r1) if and only if
for every $a, b \in X$ there exists a~unique $x \in X$ such that
$\sigma_a(\tau_{\sigma^{-1}_x(a)}(x)) = b$.
This condition is equivalent to~\eqref{cond}.
\end{proof}

\begin{proposition}
Let $(G, \cdot, \circ)$ be a~skew left brace and let $S_G$ be
the corresponding solution to YBE.
Suppose that $S_G$~is of the rack type,
then $G$ is the trivial skew left brace, and so
$S_G(x, y) = (y, x^y)$, $x, y \in G$.
\end{proposition}

\begin{proof}
Since $S_G$ is of the rack type, then
$S_G(x, y) = (y, \tau_y(x))$, $x, y \in G$,
where $\tau_y \colon G \to G$ is a bijection for all $y \in G$.
Comparing this formula with~\eqref{BraceToYBE},
we deduce that $\lambda_x = \id$ for all $x \in G$.
Hence, $x \cdot y = x \circ y$ for all $x, y \in G$,
and $G$ is trivial. Moreover,
$S_G(x, y) = (y, x^y)$ for $x, y \in G$. \qedhere
\end{proof}

\begin{example}
The groupoid $(X,*)$, where $X = \mathbb{Z}$ and $y * x = y+1$, is a~rack.
Put $S\colon X \times X \to X \times X$ by $S(x, y) = (y+1, x)$.
Then the pair $(X, S)$ is a~non-degenerate non-involutory solution to YBE.
This solution is conjugate to the solution of the rack type.
Let us find a~skew left brace corresponding to~$(X, S)$.
For this we have to construct the structure group $G(X, S)$~\cite{GV2017}.
This group is generated by elements $x_i$, $i \in \mathbb{Z}$,
and is defined by the relations
$$
x _i \circ x_j = x_{j+1}\circ  x_i, \quad i,j \in \mathbb{Z}.
$$
These relations imply  that $x_i = x_j$ for all $i, j$.
In particular $G(X, S) = \langle x_0 \rangle$
is the infinite cyclic group but the natural map $X\to G(X,S)$ is not an embedding.
Hence, $\lambda_x = \id$ for all $x \in G(X, S)$,
and so the additive operation on $G(X, S)$ is defined by the rule
$$
x \cdot y = x \circ \lambda_x^{-1}(y) = x \circ y,
$$
Thus, we get the trivial brace which defines the trivial solution $S_G = P$ on $G(X,S)$.
\end{example}

For an arbitrary Rota---Baxter group $(G,\cdot,B)$ we can prove

\begin{proposition}
Let $(G,\cdot,B)$ be a~Rota---Baxter group. Then the map
$$
S(x,y) = (y,B(y)xB(y)^{-1}), \quad x,y\in G,
$$
defines a~non-degenerate solution to YBE if and only if
$(B(b)^{-1})^{B(c)} = B(b^{B(c)})$
holds for all $b, c \in G$.
\end{proposition}

\section{Skew left multibraces}

The next definition gives a~generalization of skew left brace.

\begin{definition}
Let $k$ be a~natural number.
By {\it skew left $k$-brace} we call a $(k+1)$-groupoid
$(G, \circ_0, \circ_1, \ldots,\circ_k)$, i.\,e.,
a~non-empty set~$G$ with $k+1$ binary algebraic operations, such that

1) $(G, \circ_i)$ is a~group for all $i=0,1,\ldots,k$;

2) for $0 < i \leq k$
$$
a \circ_{i} (b \circ_{i-1} c)
 = (a \circ_{i} b) \circ_{i-1} a^{\circ_{i-1}(-1)} \circ_{i-1} (a \circ_{i} c),
$$
where $a^{\circ_{i-1}(-1)}$ is the inverse to $a$ in the group $(G, \circ_{i-1})$.
\end{definition}

A~skew left 1-brace is just the skew left brace.
By skew left multibrace we call a~skew left $k$-brace for some $k>1$.

\begin{proposition}
Let $(G, \cdot, B)$ be an RB-group and $k$ be a~natural number.
Define on the set~$G$ binary operations $\circ_1,\circ_2,\ldots,\circ_k$ as follows,
$$
x \circ_{i+1} y = x \circ_i B(x) \circ_i y \circ_i (B(x))^{\circ_i(-1)},
$$
where $\circ_0 = \cdot$. Then
$(G,\cdot,\circ_1, \circ_2,\ldots,\circ_k)$ is a~skew left $k$-brace.
\end{proposition}

\begin{proof}
It follows by Proposition~\ref{prop:Derived}.
\end{proof}

Given a skew left brace~$G$, denote by $a^{\circ(n)}$ the $n$th power of~$a$
under the product~$\circ$.

In~\cite{SV}, a~skew left brace $(S_3,\cdot,\circ)$ such that $(S_3,\cdot)$
is the symmetric group of order 6 and $(S_3,\circ)$ is the cyclic group of order 6
was constructed. The group $S_3$ is generated by two transpositions $s_1$ and $s_2$.
The following example shows that this skew left brace is
defined by a~Rota---Baxter group.

\begin{example}
a) Consider the skew left brace
$S_3(B_1) = (S_3, \cdot, \circ)$ that we can construct on~$S_3$
applying the splitting Rota---Baxter operator $B_1 \colon S_3 \to A_3$, which comes
from the decomposition $S_3 = \langle s_2 \rangle A_3$ if we put
$B_1(c a) = a^{-1}$, $c\in \langle s_2 \rangle$, $a \in A_3$, i.\,e.,
$$
B_1(s_1) = s_1 s_2,\quad
B_1(s_2) = e,\quad
B_1(s_1 s_2) = s_2 s_1, \quad
B_1(s_2 s_1) = s_1s_2, \quad
B_1(s_1 s_2 s_1) = s_2 s_1.
$$
Then the group $S_3^{(\circ)}$ is the cyclic group of order~6 with the generator $s_1$,
$$
s_1^{\circ (2)} = s_1 \circ s_1 = s_1 s_2,\quad
s_1^{\circ (3)} = s_2, \quad
s_1^{\circ (4)} = s_2 s_1, \quad
s_1^{\circ (5)} = s_1 s_2 s_1, \quad
s_1^{\circ (6)} = e.
$$
Thus, the operator $B_1$ is an endomorphism of $\mathbb{Z}_6$ acting as
$s_1\to s_1^{\circ (2)}$.

b) Consider the skew left brace $S_3(B_2) = (S_3, \cdot, \circ)$
that we can construct on $S_3$ applying the Rota---Baxter operator
$B_2 \colon S_3 \to \langle s_1 \rangle$:
$$
B_2(s_1) = s_1, \quad
B_2(s_2) = s_1, \quad
B_2(s_1 s_2 s_1) = s_1, \quad
B_2(s_2 s_1) = e, \quad
B_2(s_1 s_2) = e.
$$
Actually $B_2$ is the homomorphism from $S_3$ to its abelian subgroup $\langle s_1\rangle$,
see Proposition~\ref{prop:RB-Hom}b.
Note that the group $S_3^{(\circ)}$ is the cyclic group of order~6 with the generator~$s_2$,
$$
s_2^{\circ (2)} = s_1 s_2, \quad
s_3^{\circ (3)} = s_1, \quad
s_2^{\circ (4)} = s_2 s_1, \quad
s_5^{\circ (5)} = s_1 s_2 s_1, \quad
s_2^{\circ (6)} = e.
$$
\end{example}

This example shows that different Rota---Baxter operators
that are defined on the same group can give isomorphic skew left braces.

\begin{question}
Let $B$ and $B'$ are two Rota---Baxter operators on a~group~$G$.
Under which conditions skew left braces $G(B)$ and $G(B')$ are isomorphic?
\end{question}

\begin{example}
Above, we have considered the skew left brace $S_3(B_1) = (S_3, \cdot, \circ)$.
Also, we can define the skew left 2-brace $(S_3, \cdot, \circ_1, \circ_2)$
using the same operator $B_1$ and take $\circ_1 = \circ$.
Since $(S_3, \circ_1)$ is abelian, we have $\circ_2 = \circ_1$.
\end{example}

\section*{Acknowledgments}
Authors are grateful to I. Colazzo and L. Vendramin for their corrections and remarks about the results from~\S5.2.
Authors are also grateful to participants of the seminar ``\'{E}variste Galois'' 
at Novosibirsk State University for fruitful discussions 
and to the anonymous referee for the helpful remarks.

Valery G. Bardakov is supported by Ministry of Science and Higher Education of Russia
(agreement No. 075-02-2020-1479/1).

Vsevolod Gubarev is supported by Russian Science Foundation (project 21-11-00286).

The results of~\S5.1,~\S5.3,~\S6, and~\S7 are supported by Ministry of Science
and Higher Education of Russia (agreement No.~075-02-2020-1479/1),
while the results of~\S3,~\S4, and~\S5.2 are supported by Russian Science Foundation
(project 21-11-00286).

\bigskip

\noindent Valeriy G. Bardakov \\
Sobolev Institute of Mathematics \\
Acad. Koptyug ave. 4, 630090 Novosibirsk, Russia \\
Novosibirsk State University \\
Pirogova str. 2, 630090 Novosibirsk, Russia \\
Novosibirsk State Agrarian University \\
Dobrolyubova str., 160, 630039 Novosibirsk, Russia \\
Regional Scientific and Educational Mathematical Center of Tomsk State University \\
Lenin ave. 36, 634009 Tomsk, Russia \\
email: bardakov@math.nsc.ru

\medskip
\noindent Vsevolod Gubarev \\
Sobolev Institute of Mathematics \\
Novosibirsk State University \\
e-mail: wsewolod89@gmail.com


\begin{thebibliography}{67}
\bibitem{AnBai}
H. An, C. Bai,
From Rota-Baxter Algebras to Pre-Lie Algebras,
J. Phys. A (1) (2008), 015201, 19~p.

\bibitem{Bach16}
D. Bachiller,
Study of the algebraic structure of left braces and the Yang-Baxter equation, Ph.D. thesis, Autonomous University of Barcelona, 2016.

\bibitem{Bach}
D. Bachiller,
Solutions of the Yang-Baxter equation associated to skew left braces,
with applications to racks, J.~Knot Theory Ramif. (8) {\bf 27} (2018), 1850055.

\bibitem{Lax}
C. Bai, L. Guo, X. Ni,
Nonabelian generalized Lax pairs, the classical
Yang-Baxter equation and PostLie algebras,
Comm. Math. Phys. {\bf 297} (2010), 553--596.

\bibitem{BR}
A. Ballester-Bolinches, R. Esteban-Romero,
Triply Factorised Groups and the Structure of Skew Left Braces, 
Commun. Math. Stat., doi:10.1007/s40304-021-00239-6.

\bibitem{BG}
V.~G.~Bardakov, V. Gubarev,
Rota---Baxter operators on groups,
arXiv:2103.01848, 26~p.

\bibitem{BNY-1}
V.~G.~Bardakov, M.~V.~Neshchadim, M.~K.~Yadav,
On $\lambda$-homomorphic skew braces, 
J. Pure Appl. Algebra (6) {\bf 226} (2022), 106961.

\bibitem{Baxter}
G. Baxter,
An analytic problem whose solution follows from a~simple algebraic identity,
Pacific J.~Math. {\bf 10} (1960), 731--742.

\bibitem{Bonatto}
M. Bonatto, P. Jedli\v{c}ka,
Central nilpotency of skew braces, arXiv:2109.04389, 10~p.

\bibitem{Braga}
S.~L. de Bragan\c{c}a,
Finite Dimensional Baxter Algebras,
Stud. Appl. Math. (1) {\bf 54} (1975), 75--89.

\bibitem{Burde_41}
D. Burde, K. Dekimpe and K. Vercammen,
Affine actions on Lie groups and post-Lie algebra structures.
Linear Algebra Appl. (5) {\bf 437} (2012), 1250--1263.

\bibitem{CCD}
E. Campedel, A. Caranti, I. Del Corso,
Hopf-Galois structures on extensions of degree $p^2q$ and skew braces of order $p^2q$: 
the cyclic Sylow $p$-subgroup case, J. Algebra {\bf 556} (2020), 1165--1210.

\bibitem{CS2022}
A. Caranti, L. Stefanello, 
Skew braces from Rota--Baxter operators: A cohomological characterisation, and some examples, arXiv:2201.03936, 14~p.

\bibitem{Annulator}
F. Catino, I. Colazzo, and P.~Stefanelli,
Skew left braces with non-trivial annihilator, J. Algebra Appl.,
2019, 1950033, 23~p.

\bibitem{CedoJespersOkninski}
F. Ced\'{o}, E. Jespers, and J. Okni\'{n}ski,
Braces and the Yang-Baxter equation,
Comm. Math. Phys. (1) {\bf 327} (2014), 101--116.

\bibitem{NilpotentBraces}
F. Ced\'{o}, A. Smoktunowicz, L. Vendramin,
Skew left braces of nilpotent type,
Proc. Lond. Math. Soc. (3) {\bf 118} (2019), no. 6, 1367--1392.

\bibitem{Drinfeld}
V.~G. Drinfeld,
On some unsolved problems in quantum group theory.
In: Kulish P.P. (eds) Quantum Groups. Lecture Notes in Mathematics, vol. 1510.
Springer, Berlin, Heidelberg, 1992.

\bibitem{ESS}
P. Etingof, T. Schedler, and A.~Soloviev,
Set-theoretical solutions to the quantum Yang-Baxter equation,
Duke Math.~J. (2) {\bf 100} (1999), 169--209.

\bibitem{Goncharov2020}
M. Goncharov,
Rota---Baxter operators on cocommutative Hopf algebras,
J. Algebra {\bf 582} (2021), 39--56.

\bibitem{GV2017}
L. Guarnieri and L. Vendramin,
Skew braces and the Yang-Baxter equation,
Math. Comp. {\bf 86} (2017), 2519--2534.

\bibitem{SumOfFields}
V. Gubarev,
Rota-Baxter operators on a sum of fields,
J. Algebra Appl. (6) {\bf 19} (2020), 2050118.

\bibitem{Embedding}
V. Gubarev, P. Kolesnikov,
Embedding of dendriform algebras into Rota---Baxter algebras,
Cent. Eur. J. Math. (2) {\bf 11} (2013), 226--245.

\bibitem{GuoMonograph}
L. Guo,
An Introduction to Rota---Baxter Algebra. Surveys of Modern Mathematics, vol. 4,
International Press, Somerville (MA, USA); Higher education press, Beijing, 2012.

\bibitem{Guo2020}
L. Guo, H. Lang, and Yu. Sheng,
Integration and geometrization of Rota---Baxter Lie algebras,
Adv. Math. {\bf 387} (2021), 107834.

\bibitem{Ito}
N. It\^{o},
\"{U}ber das Produkt von zwei abelschen Gruppen,
Math. Z. {\bf 62} (1955), 400--401.

\bibitem{Jespers}
E. Jespers, \L{}. Kubat, A. Van Antwerpen, L. Vendramin, 
Factorizations of skew braces,
Math. Ann., (3-4) {\bf 375} (2019), 1649--1663.

\bibitem{JSZ}
J. Jiang, Y. Sheng, C. Zhu,
Lie theory and cohomology of relative Rota-Baxter operators,
arXiv:2108.02627, 29~p.

\bibitem{Joyce}
D.~Joyce,
A~classifying invariant of knots, the knot quandle,
J.~Pure Appl. Algebra (1) {\bf 23} (1982), 37--65.

\bibitem{Kegel}
O. Kegel,
Produkte nilpotenter Gruppen,
Arch. Math. {\bf 12} (1961), 90--93.

\bibitem{Kourovka}
E. Khukhro and V. Mazurov,
Kourovka Notebook (Unsolved Problems in Group Theory). No. 19, 2019.

\bibitem{Quasi-ideal}
A. Koch P.~J.~Truman,
Opposite skew left braces and applications,
J. Algebra, {\bf 546} (2020), 218--235.

\bibitem{Kurosh}
A.G.~Kurosh, 
Lectures of the 1969-1970 Academic Year.
Moscow: Nauka, 1974 (in Russian).

\bibitem{LYZ00}
J. Lu, M. Yan and Y. Zhu, 
On the set-theoretical Yang-Baxter equation, 
Duke Math.~J. {\bf 104} (2000), 1–18.

\bibitem{Matveev}
S. Matveev,
Distributive groupoids in knot theory,
Mat. Sb. (N.S.) {\bf 119} (161), no.~1 (9), 1982, 78--88 (in Russian).

\bibitem{Nasybullov}
T. Nasybullov,
Connections between properties of the additive
and the multiplicative groups of a~two-sided skew brace,
J. Algebra {\bf 540} (2019), 156--167.

\bibitem{OEIS}
OEIS Foundation Inc.,
The On-Line Encyclopedia of Integer Sequences, \\
http://oeis.org.

\bibitem{Orza}
J. Orza,
A construction of the free skew brace, arXiv:2002.12131, 4~p.

\bibitem{Rump2}
W. Rump,
Modules over braces, Algebra Discrete Math. {\bf 2} (2006), 127--137.

\bibitem{Rump}
W. Rump,
Braces, radical rings, and the quantum Yang-Baxter equation,
J. Algebra, (1) {\bf 307} (2007), 153--170.

\bibitem{SV}
A. Smoktunowicz and L. Vendramin,
On skew braces (with an appendix by N. Byott and L. Vendramin),
J.~Comb. Algebra (1) {\bf 2} (2018), 47--86.

\bibitem{Solovev}
A. Soloviev,
Non-unitary set-theoretical solutions to the quantum
Yang-Baxter equation,
Math. Res. Lett. (5-6) {\bf 7} (2000), 577--596.

\bibitem{Tsang}
C. Tsang,
Hopf---Galois structures of isomorphic-type on a non-abelian
characteristically simple extension, Proc. Amer. Math. Soc.
(12) {\bf 147} (2019), 5093--5103.

\bibitem{Tsang2}
C. Tsang,
Finite skew braces with isomorphic non-abelian
characteristically simple additive and circle groups,
J. Group Theory, 
https://doi.org/10.1515/jgth-2021-0044.

\bibitem{TsangChao}
C. Tsang, Q. Chao,
On the solvability of regular subgroups in the holomorph of a~finite solvable group,
Int. J. Algebra Comput. (2) {\bf 30} (2020), 253--265.

\bibitem{Tricomi}
F.~G.~Tricomi,
On the finite Hilbert transformation.
Quart. J. Math. {\bf 2} (1951) 199--211.

\bibitem{Vallette2007}
B. Vallette,
Homology of generalized partition posets, J. Pure Appl. Algebra (2) {\bf 208} (2007), 699--725.

\bibitem{Problems}
L. Vendramin,
Problems on skew left braces, Adv. Group Theory Appl. {\bf 7} (2019), 15--37.
\end{thebibliography}
\end{document}